\documentclass{gen-j-l}%
\usepackage{amssymb}
\usepackage{amsfonts}
\usepackage{amsmath}
\usepackage{graphicx}%
\newtheorem{theorem}{Theorem}[section]
\newtheorem{lemma}[theorem]{Lemma}
\theoremstyle{definition}
\newtheorem{definition}[theorem]{Definition}
\newtheorem{example}[theorem]{Example}

\theoremstyle{remark}
\newtheorem{remark}[theorem]{Remark}
\numberwithin{equation}{section}
\theoremstyle{plain}

\newtheorem{corollary}{Corollary}

\copyrightinfo{2001}{enter name of copyright holder}
\begin{document}
\title{On the global behavior of linear flows}
\author{Fritz Colonius}
\address{Institut f\"{u}r Mathematik, Universit\"{a}t Augsburg, Germany\\
fritz.colonius@uni-a.de}

\begin{abstract}
For linear flows on vector bundles, it is analyzed when subbundles in the
Selgrade decomposition yield chain transitive subsets for the induced flow on
the associated Poincar\'{e} sphere bundle.

\end{abstract}
\maketitle

\section{Introduction}

For linear flows on vector bundles, Selgrade's theorem describes the
decomposition into subbundles obtained from the chain recurrent components of
the induced flow on the projective bundle. This coincides with the finest
decomposition into exponentially separated subbundles but has the advantage
that it provides an intrinsic characterization using chain transitivity
properties in the projective bundle. The present paper complements this point
of view by considering recurrence properties of the linear flow. This is based
on the compactification provided by the\ construction of the Poincar\'{e}
sphere from the global theory of ordinary differential equations going back to
Poincar\'{e} \cite{Poin}; cf., e.g., Perko \cite[Section 3.10]{Perko}.

We will consider linear flows on a vector bundle $\pi:\mathcal{V}\rightarrow
B$ over a compact chain transitive metric space $B,$ that is, a continuous
flow $\Phi$ on $\mathcal{V}$ preserving fibers such that the induced maps on
the fibers $\mathcal{V}_{b}$ are linear. Classical examples of such flows are
given by linear differential equations with almost periodic coefficients and
by linearized flows over a compact invariant chain transitive set. We refer to
Selgrade \cite{Selg75}, Salamon and Zehnder \cite{SalZ88}, Bronstein and
Kopanskii \cite{BroK94}, Johnson, Palmer and Sell \cite{JoPS87}, and Colonius
and Kliemann \cite[Chapter 5]{ColK00} for the theory of linear flows, and to
Alves and San Martin \cite{AlvSM16} for generalizations to principal bundles.
Blumenthal and Latushkin \cite{BluL19} generalize Selgrade's theorem to linear
skew product semiflows on separable Banach bundles.

Selgrade's theorem (\cite[Theorem 5.2.5]{ColK00}) states that the induced flow
$\mathbb{P}\Phi$ on the projective bundle $\mathbb{P}\mathcal{V}$ has finitely
many chain recurrent components $_{\mathbb{P}}M_{1},\ldots,\,_{\mathbb{P}%
}M_{l}$ (this coincides with the finest Morse decomposition) and $1\leq l\leq
d:=\dim\,\mathcal{V}_{b},\,b\in B.$ Every chain recurrent component
$_{\mathbb{P}}M_{i}$ defines an invariant subbundle $\mathcal{V}%
_{i}=\mathbb{P}^{-1}\left(  _{\mathbb{P}}M_{i}\right)  $ of $\mathcal{V}$ and
the following decomposition into a Whitney sum holds:
\begin{equation}
\mathcal{V}=\mathcal{V}_{1}\oplus\cdots\oplus\mathcal{V}_{l}.
\label{Morsedec1}%
\end{equation}
Note that this refines the subbundle decomposition obtained by exponential
dichotomies. It is clear that in \textquotedblleft stable\textquotedblright%
\ and \textquotedblleft unstable\textquotedblright\ subbundles no recurrence
properties can be expected. This is different in the nonhyperbolic case, where
a \textquotedblleft central\textquotedblright\ subbundle is present: If for an
autonomous ordinary differential equation $\dot{x}=Ax$ the matrix
$A\in\mathbb{R}^{d\times d}$ has $0$ as the only eigenvalue with vanishing
real part and $0$ has geometric multiplicity $1$, then the eigenspace consists
of equilibria, hence it is chain transitive. If, e.g., $0$ has higher
geometric multiplicity the behavior \textquotedblleft at
infinity\textquotedblright\ has to be taken into account. This can be done
using the Poincar\'{e} sphere which is obtained by attaching to the sphere
$\mathbb{S}^{d}$ in $\mathbb{R}^{d+1}$ a copy of $\mathbb{R}^{d}$ at the north
pole and by taking the central projection from the origin in $\mathbb{R}%
^{d+1}$ to the northern hemisphere $\mathbb{S}^{d,+}$ of $\mathbb{S}^{d}$.

Based on this construction, the main results of this paper are Theorem
\ref{Theorem_Poincare1} and Theorem \ref{Theorem_Poincare2} showing when for
linear flows subbundles $\mathcal{V}_{i}$ yield chain transitive sets on the
appropriately defined Poincar\'{e} sphere bundle. A sufficient condition is
based on a convenient spectral concept, the Morse spectrum, which is
constructed via chains in the chain recurrent components on the projective
bundle $\mathbb{P}\mathcal{V}$. Here we instead consider chain recurrent
components on the sphere bundle $\mathbb{S}\mathcal{V}$ and an analogously
defined Morse spectrum.

The contents of this paper are as follows. Section \ref{Section2} recalls some
notation for linear flows on vector bundles and analyzes the relation between
the chain recurrent components for the induced flows on the projective bundle
and the sphere bundle. Section \ref{Section3} recalls the Morse spectrum and
Section \ref{Section4} shows the main results on the projected flow on the
Poincar\'{e} sphere bundle. In Section \ref{Section5} the results are
illustrated for linear autonomous differential equations and for bilinear
control systems (cf. Elliott \cite{Elliott}, Jurdjevic \cite{Jurd97}).

\section{Linear flows on vector bundles\label{Section2}}

This section introduces some notation and describes for\ (continuous) linear
flows on a vector bundle $\pi:\mathcal{V}\rightarrow B$ over a metric space
$B$ the relation between the chain recurrent components on the projective
bundle and the sphere bundle.

We consider vector bundles of the following form (for the following cf.
Salamon and Zehnder \cite[Appendix]{SalZ88}).

\begin{definition}
\label{Definition1}Let $H$ be a finite-dimensional Hilbert space over the
reals $\mathbb{R}$ and let $B$ be a compact metric space. A vector bundle is a
continuous map $\pi:\mathcal{V}\rightarrow B$ where $\mathcal{V}$ is a
topological space with the following properties:\smallskip

(i) There exist a finite open cover $\left\{  U_{a},\;a\in A\right\}  $ and
homeomorphisms $\varphi_{a}:\pi^{-1}(U_{a})\allowbreak\rightarrow U_{a}\times
H$ such that $pr_{1}\circ\varphi_{a}=\pi$, where $pr_{1}:U_{a}\times
H\rightarrow U_{a}$ denotes the projection onto the first component.

(ii) The maps $L_{ba}:H\rightarrow H$ defined by $\varphi_{b}\circ\varphi
_{a}^{-1}(b,h)=(b,L_{ba}h)$ are linear for every $b\in U_{a}\cap U_{b}$.
\end{definition}

Without loss of generality one may assume that $\varphi_{a}$ extends to a
homeomorphism of $\overline{\pi^{-1}(U_{a})}$ onto $\overline{U_{a}}\times H$.
The fibers $\mathcal{V}_{b}$ can be given a Hilbert space structure by%
\begin{equation}
c_{1}v_{1}+c_{2}v_{2}=\varphi_{a}^{-1}(b,c_{1}x_{1}+c_{2}x_{2})\text{ for
}c_{i}\in\mathbb{R},\,v_{i}=\varphi_{a}^{-1}(b,x_{i})\in\mathcal{V}_{b},i=1,2,
\label{linear}%
\end{equation}
and using the continuous and positive definite linear form on $\mathcal{V}$
defined by%
\begin{equation}
\left\langle v,v^{\prime}\right\rangle =%
%TCIMACRO{\dsum \limits_{a\in A}}%
%BeginExpansion
{\displaystyle\sum\limits_{a\in A}}
%EndExpansion
d(\pi(v),B\setminus U_{a})\left\langle v,v^{\prime}\right\rangle _{a}\text{
for }\pi(v)=\pi(v^{\prime}). \label{form}%
\end{equation}
For $v=\varphi_{a}^{-1}(b,x)$ and $v^{\prime}=\varphi_{a}^{-1}(b^{\prime
},x^{\prime})$ with $b,b^{\prime}\in U_{a}$ let
\begin{equation}
d_{a}(v,v^{\prime})=\max\left\{  d(b,b^{\prime}),\parallel x-x^{\prime
}\parallel_{a}\right\}  . \label{d_a}%
\end{equation}
For any sequence $v_{0},...,v_{m}\in\mathcal{V}$ we define
\[
\rho(v_{0},...,v_{m})=\max\left\{  \sum_{j=1}^{m}d_{a_{j}}(v_{j-1}%
,v_{j})\left\vert \pi(v_{j-1}),\pi(v_{j})\in U_{a_{j}}\text{ with }a_{j}\in
A\right.  \right\}  ,
\]
where the maximum over the empty set is by definition $+\infty$. Then the
distance function
\begin{equation}
d(v,v^{\prime})=\inf\left\{  \rho(v_{0},...,v_{m})\left\vert \,m\in
\mathbb{N},\,v_{j}\in\mathcal{V},\;v_{0}=v,\;v_{m}=v^{\prime}\right.
\right\}  \label{d}%
\end{equation}
defines a metric on $\mathcal{V}$ that is compatible with the original
topology. The corresponding projective bundle is $\mathbb{P}\mathcal{V}%
=(\mathcal{V}\setminus Z)/\sim$, where $Z$ is the zero section and
\thinspace$v\sim v^{\prime}$ if $\pi(v)=\pi(v^{\prime})$. This is a fiber
bundle $\mathbb{P}\pi:\mathbb{P}\mathcal{V}\rightarrow B$ with local
trivializations $\mathbb{P}\varphi_{a}:\mathbb{P}\pi^{-1}(U_{a})\rightarrow
U_{a}\times\mathbb{P}H$. The (unit) sphere bundle is $\mathbb{S}\pi
:\mathbb{S}\mathcal{V}:=\{v\in\mathcal{V}\left\vert \left\Vert v\right\Vert
=1\right.  \}\rightarrow B$. The metric spaces $\mathbb{P}\mathcal{V}$ and
$\mathbb{S}\mathcal{V}$ are compact.

\begin{definition}
\label{continuouslinflow}A \textit{linear flow} $\Phi$ on a vector bundle
$\pi:\mathcal{V}\rightarrow B$ is a flow $\Phi$ on $\mathcal{V}$ such that for
all $\alpha\in\mathbb{R}\,$ and $v,w\in\mathcal{V}$ with $\pi(v)=\pi(w)$ and
$t\in\mathbb{R}$ one has
\[
\pi\left(  \Phi(t,\,v)\right)  =\pi\left(  \Phi(t,\,w)\right)  ,\,\Phi
(t,\alpha(v+w))=\alpha\Phi(t,v)+\alpha\Phi(t,\,w)\,.\,
\]

\end{definition}

It is convenient to denote the induced flow on the base space $B$ by $b\cdot
t,t\in\mathbb{R}$. The restriction of $\Phi(t,\cdot)$ to a fiber
$\mathcal{V}_{b}$ is a linear map to $\mathcal{V}_{b\cdot t}$ denoted by
$\Phi_{b,t}$ with $\left\Vert \Phi_{b\cdot t}(v)\right\Vert =\left\Vert
\Phi(t,v)\right\Vert $ for $v\in\mathcal{V}_{b}$. This is an isomorphism with
inverse $\Phi_{b\cdot t,-t}$. Furthermore, $\Phi$ induces a flow
$\mathbb{S}\Phi$ on the sphere bundle satisfying $\mathbb{S}\Phi
(t,v)=\frac{\Phi(t,v)}{\left\Vert \Phi(t,v)\right\Vert }$ for all $t,v$ and a
flow on the projective bundle denoted by $\mathbb{P}\Phi$.

For $\varepsilon,\,T>0$ an ($\varepsilon,T)$-chain $\zeta$ for a flow $\Psi$
on a metric space $X$ from $x$ to $y$ is given by $n\in\mathbb{N}\mathbf{,}$
$T_{0},\ldots,T_{n-1}\geq T$, and $x_{0}=x,\ldots,x_{n}=y\in X$ with
$d(\Psi(T_{i},x_{i}),x_{i+1})<\varepsilon$ for $i=0,\ldots,n-1$. For $x\in X$
the (forward) chain limit set is%
\[
\Omega(x)=\{y\in X\left\vert \forall\varepsilon,T>0\,\exists(\varepsilon
,T)\text{ chain from }x\text{ to }y\right.  \}.
\]
A set is called chain transitive if $y\in\Omega(x)$ for all $x,y$ in this set.
A chain recurrent component $M$ is a maximal chain transitive set (on a
compact metric space this are the connected components of the chain recurrent
set). The monograph Alongi and Nelson \cite{AlonN07} provides detailed proofs
around chains; cf. also Robinson \cite{Robin98}.

Next we analyze the relations between the chain recurrent components for the
induced flows $\mathbb{P}\Phi$ on the projective bundle $\mathbb{P}%
\mathcal{V}$ and $\mathbb{S}\Phi$ on the sphere\ bundle $\mathbb{S}%
\mathcal{V}$.

\begin{lemma}
\label{Lemma_sphere}(i) Let $v,w\in\mathbb{S}\mathcal{V}$. If on
$\mathbb{P}\mathcal{V}$ the point $\mathbb{P}w$ is in $\Omega(\mathbb{P}v)$,
then on $\mathbb{S}\mathcal{V}$ at least one of the points $w$ or $-w$ is in
$\Omega(v)$.

(ii) Let $_{\mathbb{S}}M$ be a chain recurrent component on $\mathbb{S}%
\mathcal{V}$. Then the projection of $_{\mathbb{S}}M$ to $\mathbb{P}%
\mathcal{V}$ is contained in a chain recurrent component $_{\mathbb{P}}M$.

(iii) Consider a chain recurrent component $_{\mathbb{P}}M$ on $\mathbb{P}%
\mathcal{V}$. Suppose that there is $v_{0}\in\mathbb{S}\mathcal{V}$ such that
$\mathbb{P}v_{0}\in\,_{\mathbb{P}}M$ and $-v_{0}\in\Omega(v_{0})$. Then
$\{v\in\mathbb{S}\mathcal{V}\left\vert \mathbb{P}v\in\,_{\mathbb{P}}M\right.
\}$ is a chain recurrent component on $\mathbb{S}\mathcal{V}$.
\end{lemma}

\begin{proof}
We will frequently use the following elementary fact: If $w\in\Omega(v)$ for
the induced system on $\mathbb{S}\mathcal{V}$, then $-w\in\Omega(-v)$.

(i) For $\varepsilon,T>0$, let an $(\varepsilon,T)$-chain given by
$\mathbb{P}v_{0}=\mathbb{P}v,\mathbb{P}v_{1},\ldots,\mathbb{P}v_{n}%
=\mathbb{P}w$ and $T_{0},\ldots,\allowbreak T_{n-1}\geq T$ with $d(\mathbb{P}%
\Phi(T_{i},v_{i}),\mathbb{P}v_{i+1})<\varepsilon$. In order to construct an
$(\varepsilon,T)$-chain on $\mathbb{S}\mathcal{V}$ note that $d(\mathbb{S}%
\Phi(T_{0},v_{0}),v_{1})<\varepsilon$ or $d(\mathbb{S}\Phi(T_{0},v_{0}%
),-v_{1})<\varepsilon$. In the first case define $v_{1}^{1}=v_{1}$ and in the
second case define $v_{1}^{1}=-v_{1}$. Proceeding in this way, one finds an
$(\varepsilon,T)$-chain $v_{0},v_{1}^{1},\ldots,v_{n}^{1}$ from $v$ to $w$ or
$-w$.

Assertion (ii) is immediate from the definitions. Concerning assertion (iii)
let $v_{1},v_{2}\in\{v\in\mathbb{S}\mathcal{V}\left\vert \mathbb{P}%
v\in\,_{\mathbb{P}}M\right.  \}$. We have to show that $v_{2}\in\Omega(v_{1}%
)$. Since $\mathbb{P}v_{1},\mathbb{P}v_{2}\in\,_{\mathbb{P}}M$ it follows that
$v_{2}$ or $-v_{2}$ is in $\Omega(v_{1})$. In the first case we are done. In
the second case it follows that $v_{2}\in\Omega(-v_{1})$ and, by our
assumption, $v_{0}\in\Omega(-v_{0})$. Now $\mathbb{P}(-v_{1})=\mathbb{P}%
v_{1}\in\,_{\mathbb{P}}M$ and $\mathbb{P}v_{0}\in\,_{\mathbb{P}}M$ imply that
$\mathbb{P}v_{0}\in\Omega(\mathbb{P}(-v_{1}))$, hence \thinspace(a) $v_{0}%
\in\Omega(-v_{1})$ or (b) $-v_{0}\in\Omega(-v_{1})$. We claim also (b) implies
that $v_{0}\in\Omega(-v_{1})$. In fact, $-v_{0}\in\Omega(-v_{1})$ implies
$v_{0}\in\Omega(-v_{0})\subset\Omega(-v_{1})$ showing the claim.

Since $\mathbb{P}v_{0},\mathbb{P}v_{2}\in\,_{\mathbb{P}}M$ it follows that
$v_{2}\in\Omega(v_{0})$ or $-v_{2}$ is in $\Omega(v_{0})$. In the first case,
one has%
\[
v_{2}\in\Omega(v_{0})\subset\Omega(-v_{0})\subset\Omega(v_{1})
\]
and in the second case $v_{2}\in\Omega(-v_{0})\subset\Omega(v_{1})$.
\end{proof}

We get the following result characterizing the relation between the chain
recurrent components $_{\mathbb{P}}M_{1},\ldots,\,_{\mathbb{P}}M_{l},1\leq
l\leq d$, on the projective bundle and the chain recurrent components on the
sphere bundle.

\begin{theorem}
\label{Theorem_sphere}(i) The set $\{v\in\mathbb{S}\mathcal{V}\left\vert
\mathbb{P}v\in\,_{\mathbb{P}}M_{i}\right.  \}$ is the unique chain recurrent
component on $\mathbb{S}\mathcal{V}$ which projects to $_{\mathbb{P}}M_{i}$ if
and only if there is $v_{0}\in\mathbb{S}\mathcal{V}$ with $\mathbb{P}v_{0}%
\in\,_{\mathbb{P}}M_{i}$ such that $-v_{0}\in\Omega(v_{0})$ for the system on
the sphere bundle $\mathbb{S}\mathcal{V}$.

(ii) For every chain recurrent component $_{\mathbb{P}}M_{i}$ there are at
most two chain recurrent components $_{\mathbb{S}}M^{+}$ and $_{\mathbb{S}%
}M^{-}$ on $\mathbb{S}\mathcal{V}$ such that%
\begin{equation}
\{v\in\mathbb{S}\mathcal{V}\left\vert \mathbb{P}v\in\,_{\mathbb{P}}%
M_{i}\right.  \}=\,_{\mathbb{S}}M^{+}\cup\,_{\mathbb{S}}M^{-}\text{, and then
}_{\mathbb{S}}M^{+}=-\,_{\mathbb{S}}M^{-}. \label{sphere_1}%
\end{equation}

(iii) There are $l_{1}$ chain recurrent components on $\mathbb{S}\mathcal{V}$
denoted by $_{\mathbb{S}}M_{1},\ldots,\allowbreak\,_{\mathbb{S}}M_{l_{1}}$
with $1\leq l_{1}\leq2l\leq2d$.
\end{theorem}

\begin{proof}
(i) Suppose that there is $\mathbb{P}v_{0}\in\,_{\mathbb{P}}E_{i}$ such that
$-v_{0}\in\Omega(v_{0})$. Then Lemma \ref{Lemma_sphere}(ii) shows the
assertion. The converse is obvious.

(ii) Fix $v\in\mathbb{S}\mathcal{V}$ with $\mathbb{P}v\in\,_{\mathbb{P}}M_{i}%
$. For any $w\in\mathbb{S}\mathcal{V}$ with $\mathbb{P}w\in\,_{\mathbb{P}%
}M_{i}$ and $\varepsilon,T>0$, one finds as in Lemma \ref{Lemma_sphere}(i) an
$(\varepsilon,T)$-chain $v_{1}^{1},\ldots,v_{n}^{1}$ on $\mathbb{S}%
\mathcal{V}$ with times $T_{0},\ldots,T_{n-1}\geq T$ from $v$ to $w$ or $-w$.
Define%
\[
A^{\pm}:=\left\{  v\in\mathbb{S}\mathcal{V}\left\vert \mathbb{P}%
v\in\,_{\mathbb{P}}M_{i}\text{ and }\pm v\in\Omega(v_{0})\text{ and }v_{0}%
\in\Omega(\pm v)\right.  \right\}  .
\]
If $A^{\pm}\not =\varnothing$ there are chain recurrent components
$_{\mathbb{S}}M^{\pm}$ with $A^{\pm}\subset\,_{\mathbb{S}}M^{\pm}$. For every
point $v$ with $\mathbb{P}v\in\,_{\mathbb{P}}M_{i}$ one has $v\in\Omega
(v_{0})$ or $-v\in\Omega(v_{0})$, and $v_{0}\in\Omega(v)$ or $v_{0}\in
\Omega(-v)$. If there is $v\in\Omega(v_{0})$ and $v_{0}\in\Omega(-v)$ it
follows that $-v_{0}\in\Omega(v)\subset\Omega(v_{0})$. Then (i) implies the assertion.

Hence we may assume that one of the following properties hold for every
$v\in\mathbb{S}\mathcal{V}$ with $\mathbb{P}v\in\,_{\mathbb{P}}M_{i}$:
$v\in\Omega(v_{0})$ and $v_{0}\in\Omega(v)$, or $-v\in\Omega(v_{0})$ and
$v_{0}\in\Omega(-v)$. It follows that
\[
\{v\in\mathbb{S}\mathcal{V}\left\vert \mathbb{P}v\in\,_{\mathbb{P}}%
M_{i}\right.  \}\subset A^{+}\cup A^{-}\subset\,_{\mathbb{S}}M^{+}%
\cup\,_{\mathbb{S}}M^{-}%
\]
(one of the sets $_{\mathbb{S}}M^{+}$ and $_{\mathbb{S}}M^{-}$may be void). By
Lemma \ref{Lemma_sphere}(ii) the projections of $\,_{\mathbb{S}}M^{+}$ and
$_{\mathbb{S}}M^{-}$ are contained in $_{\mathbb{P}}M_{i}$. This proves the
first equality in (\ref{sphere_1}). The same arguments with $-v$ instead of
$v$ implies the second equality.

(iii) This is a consequence of assertion (ii).
\end{proof}

Next we lift chains on $\mathbb{S}\mathcal{V}$ to chains on $\mathcal{V}$.

\begin{definition}
\label{Definition_lift}Let $\zeta$ be an $(\varepsilon,T)$-chain on
$\mathbb{S}\mathcal{V}$ given by $n\in\mathbb{N},\ v_{0},\allowbreak
\ldots,\allowbreak v_{n}\in\mathbb{S}\mathcal{V}$ and $T_{0},\ldots
,T_{n-1}\geq T$ with
\[
d(\mathbb{S}\Phi(T_{i},v_{i}),v_{i+1})<\varepsilon\text{ }%
\,\text{for\thinspace}\,i=0,\ldots,n-1.
\]
The lift of $\zeta$ to $\mathcal{V}$ is the chain $\widehat{\zeta}$ given by
$n\in\mathbb{N},\,T_{0},\ldots,T_{n-1}\geq T$ and%
\[
w_{0}=v_{0},w_{i}=\prod_{j=0}^{i-1}\left\Vert \Phi(T_{j},v_{j})\right\Vert
v_{i}\text{ for }i=1,\ldots,n.
\]
Furthermore, $\alpha\widehat{\zeta}$ for $\alpha>0$ denotes the chain on
$\mathcal{V}$ given by $n\in\mathbb{N},\,T_{0},\ldots,T_{n-1}\geq T$ and
$\alpha w_{i}$ for $i=0,\ldots,n$.
\end{definition}

Note the following lemma.

\begin{lemma}
\label{Lemma1}Let $\zeta$ be an $(\varepsilon,T)$-chain on $\mathbb{S}%
\mathcal{V}$ as above and define $C_{-1}=1$ and $C_{i}=\prod_{j=0}%
^{i}\left\Vert \Phi(T_{j},v_{j})\right\Vert $ for $i=0,\ldots,n-1$. Then for
$\varepsilon>0$ small enough the lift $\alpha\widehat{\zeta},\alpha>0$, to
$\mathcal{V}$ satisfies $\left\Vert w_{i}\right\Vert =\alpha C_{i-1}$ and
$\Phi(T_{i},\alpha w_{i})=\alpha C_{i}\mathbb{S}\Phi(T_{i},v_{i})$ for all $i$.
\end{lemma}

\begin{proof}
The formula for $\left\Vert \alpha w_{i}\right\Vert $ obviously holds.
Together with linearity of $\Phi$ this yields%
\[
\Phi(T_{i},\alpha w_{i})=\alpha\prod_{j=0}^{i}\left\Vert \Phi(T_{j}%
,v_{j})\right\Vert \frac{\Phi(T_{i},v_{i})}{\left\Vert \Phi(T_{i}%
,v_{i})\right\Vert }=\alpha C_{i}\mathbb{S}\Phi(T_{i},v_{i}).
\]

\end{proof}

\section{The Morse spectrum\label{Section3}}

For linear flows on vector bundles, a number of spectral notions and their
relations have been considered; cf., e.g., Sacker and Sell \cite{SacS},
Johnson, Palmer, and Sell \cite{JoPS87}, Gr\"{u}ne \cite{Grue00}, Braga Barros
and San Martin \cite{BraSM07}, Kawan and Stender \cite{KawS12}. An appropriate
spectral notion in the present context is provided by the Morse spectrum
defined as follows (cf. Colonius and Kliemann \cite{ColK96}, and San Martin
and Seco \cite{SanS10} and Bochi \cite[Section 4]{Bochi18} for generalizations).

For $\varepsilon,\,T>0$ let an ($\varepsilon,T)$-chain $\zeta$ of
$\mathbb{P}\Phi$ be given by $n\in\mathbb{N}\mathbf{,}$ $T_{0},\ldots
,T_{n-1}\geq T$, and $p_{0},\ldots,p_{n}\in\mathbb{P}\mathcal{V}$ with
$d(\mathbb{P}\Phi(T_{i},p_{i}),p_{i+1})<\varepsilon$ for $i=0,\ldots,n-1.$
With total time $\sigma=\sum_{i=0}^{n-1}T_{i}$ let the exponential growth rate
of $\zeta$ be
\[
\lambda(\zeta):=\frac{1}{\sigma}\left(  \sum_{i=0}^{n-1}\log\left\Vert
\Phi(T_{i},v_{i})\right\Vert -\log\left\Vert v_{i}\right\Vert \right)  \text{
with }v_{i}\in\mathbb{P}^{-1}\,p_{i}.
\]
Define the Morse spectrum of a subbundle $\mathcal{V}_{i}$ generated by
$_{\mathbb{P}}M_{i}$ by
\[
\mathbf{\Sigma}_{Mo}(\mathcal{V}_{i})=\left\{
\begin{array}
[c]{c}%
\lambda\in\mathbb{R}\text{,\ there\ are}\;\varepsilon^{k}\rightarrow
0,\,T^{k}\rightarrow\infty\;\,\text{and}\,\;\\
(\varepsilon^{k},\,T^{k})\text{-chains}\mathrm{\;}\zeta^{k}\;\text{in}%
\;_{\mathbb{P}}M_{i}\;\text{with}\mathrm{\;}\lambda(\zeta^{k})\rightarrow
\lambda\;\text{as}\;k\rightarrow\infty
\end{array}
\right\}  .
\]
For a chain recurrent component $_{\mathbb{S}}M_{j},j\in\{1,\ldots,l_{1}\}$,
in $\mathbb{S}\mathcal{V}$ let the cone bundle be%
\[
\mathcal{V}_{j}^{+}:=\{\alpha v\in\mathcal{V}\left\vert \alpha>0\text{ and
}v\in\,_{\mathbb{S}}M_{j}\right.  \},
\]
and define the Morse spectrum $\Sigma_{Mo}(\mathcal{V}_{j}^{+})$ analogously
via chains in $\,_{\mathbb{S}}M_{j}$.

Every Lyapunov exponent is contained in some $\mathbf{\Sigma}_{Mo}%
(\mathcal{V}_{i},\Phi)$ and each $\mathbf{\Sigma}_{Mo}(\mathcal{V}_{i},\Phi)$
is a compact interval with boundary points corresponding to Lyapunov exponents
for ergodic measures. It suffices to consider periodic chains in the
definition of the Morse spectrum; cf. Colonius and Kliemann \cite[Chapter
5]{ColK00} for these claims. The spectral intervals $\mathbf{\Sigma}%
_{Mo}(\mathcal{V}_{i})$ need not be disjoint (in particular, there may exist
two \textquotedblleft central\textquotedblright\ subbundles with $0$ in the
Morse spectrum); cf. Salamon and Zehnder \cite[Example 2.14]{SalZ88} and also
Example \ref{MORSE:ex47}.

The following lemma shows that the Morse spectra on the sphere bundle coincide
with the Morse spectra on the projective bundle.

\begin{lemma}
\label{Proposition_Morse}If $_{\mathbb{S}}M_{j}$ is a chain recurrent
component on $\mathbb{S}\mathcal{V}$ that projects to a chain recurrent
component $_{\mathbb{P}}M_{i}$ in $\mathbb{P}\mathcal{V}$, then the Morse
spectra of the cone bundle $\mathcal{V}_{j}^{+}$ generated by $_{\mathbb{S}}M$
and of the subbundle $\mathcal{V}_{i}$ generated by $_{\mathbb{P}}M_{i}$
coincide,
\[
\Sigma_{Mo}(\mathcal{V}_{j}^{+})=\Sigma_{Mo}(\mathcal{V}_{i}).
\]

\end{lemma}

\begin{proof}
The inclusion \textquotedblleft$\subset$\textquotedblright\ holds since every
$(\varepsilon,T)$-chain in $_{\mathbb{S}}M_{j}$ projects to an $(\varepsilon
,T)$-chain in $_{\mathbb{P}}M_{i}$ with the same exponent. For the converse,
consider an $(\varepsilon,T)$-chain $\zeta$ in $_{\mathbb{P}}M_{i}$ given by
$T_{0},\ldots,T_{n-1}\geq T$ and points $\mathbb{P}v_{0},\ldots,\mathbb{P}%
v_{n}$. Then we can write%
\[
\lambda(\zeta)=\frac{1}{\sigma}\left(  \sum_{i=0}^{n-1}\log\left\Vert
\Phi(T_{i},\pm v_{i})\right\Vert -\log\left\Vert \pm v_{i}\right\Vert \right)
,
\]
where, by (\ref{sphere_1}), we may assume that $\pm v_{i}\in\,_{\mathbb{S}%
}M^{+}\cup\,_{\mathbb{S}}M^{-}\subset\mathbb{S}\mathcal{V}$. Each of the
chains constructed in the proof of Lemma \ref{Lemma_sphere}(i) has chain
exponent $\lambda(\zeta)$.
\end{proof}

Lemma \ref{Lemma1} implies for the lift of a chain $\zeta$ that $\left\Vert
w_{n}\right\Vert =\prod_{i=0}^{n-1}\left\Vert \Phi(T_{j},v_{j})\right\Vert
=e^{\lambda(\zeta)\sigma}$. Hence $\lambda(\zeta)>0$ and $\lambda(\zeta)<0$
imply $\left\Vert w_{n}\right\Vert >1$ and $\left\Vert w_{n}\right\Vert <1$, resp.

\section{The flow on the Poincar\'{e} sphere bundle\label{Section4}}

This section analyzes when a subbundle $\mathcal{V}_{i}$ yields a chain
transitive set on the Poincar\'{e} sphere bundle. In particular, a sufficient
condition is derived in terms of the Morse spectrum.

For a vector bundle $\pi:\mathcal{V}\rightarrow B$ the projection to the
Poincar\'{e} sphere can be defined in the following way (cf. Perko
\cite[Section 3.10]{Perko} for the case in $\mathbb{R}^{d}$). Define a Hilbert
space by $H^{1}=H\times\mathbb{R}$ with inner product%
\[
\left\langle (h_{1},\alpha_{1}),(h_{2},\alpha_{2})\right\rangle =\left\langle
h_{1},h_{2}\right\rangle _{H}+\alpha_{1}\alpha_{2}\text{ for }(h_{1}%
,\alpha_{1}),(h_{2},\alpha_{2})\in H\times\mathbb{R},
\]
and consider the corresponding vector bundle $\pi^{1}:\mathcal{V}%
^{1}=\mathcal{V}\times\mathbb{R}\rightarrow B$ with $\pi^{1}(v,r)=\pi(v)$
using the local trivializations $\varphi_{a}^{1}:(\pi^{1})^{-1}(U_{a}%
)\rightarrow U_{a}\times H^{1},a\in A$, defined by%
\[
\varphi_{a}^{1}(v,r)=(\varphi_{a}(v),r)=(b,(x,r))\text{ for }(v,r)\in
\mathcal{V}^{1}=\mathcal{V}\times\mathbb{R}\text{ and }(b,x)=\varphi_{a}(v).
\]
Define a map $\pi_{P}:\mathcal{V}\rightarrow\mathbb{S}\mathcal{V}^{1}$ by%
\begin{equation}
\pi_{P}(v)=\left(  \frac{v}{\left\Vert (v,1)\right\Vert },\frac{1}{\left\Vert
(v,1)\right\Vert }\right)  \text{ for }v\in\mathcal{V}. \label{S_local2}%
\end{equation}
Then $\pi_{P}$ is a homeomorphism onto the open upper hemisphere of
$\mathbb{S}\mathcal{V}^{1}$%
\[
\mathbb{S}_{P}^{+}\mathcal{V}:=\{(v,r)\in\mathcal{V}^{1}\left\vert \left\Vert
(v,r)\right\Vert =1\text{ and }r>0\right.  \}=\left\{  \left.  \pi
_{P}(v)\right\vert v\in\mathcal{V}\right\}  .
\]
We call $\mathbb{S}\mathcal{V}^{1}$ the Poincar\'{e} sphere bundle. Denote by
$\mathbb{S}_{P}^{0}\mathcal{V}$ the set of all $(v,0)\in\mathbb{S}%
\mathcal{V}^{1}$, the \textquotedblleft equator\textquotedblright\ of the
Poincar\'{e} sphere bundle, and define a homeomorphism%
\[
e:\mathbb{S}\mathcal{V\rightarrow}\mathbb{S}_{P}^{0}\mathcal{V}:v\mapsto
(v,0)).
\]
The closed upper hemisphere bundle can be written as the disjoint union
$\overline{\mathbb{S}_{P}^{+}\mathcal{V}}=\mathbb{S}_{P}^{+}\mathcal{V}%
\cup\mathbb{S}_{P}^{0}\mathcal{V}$.

A linear flow $\Phi(t,v)$ on $\mathcal{V}$ is extended to a linear flow on
$\mathcal{V}^{1}$ by $\Phi^{1}(t,(v,r))=(\Phi(t,v),r)$ and every trajectory
$\mathbb{S}\Phi(t,v)=\frac{\Phi(t,v)}{\left\Vert \Phi(t,v)\right\Vert }%
,t\in\mathbb{R}$, on $\mathbb{S}\mathcal{V}$ is mapped onto
\[
e(\mathbb{S}\Phi(t,v))=\mathbb{S}\Phi^{1}(t,(v,0))=(\mathbb{S}\Phi
(t,v),0),\,t\in\mathbb{R}.
\]

\begin{lemma}
The map $e$ is a conjugacy between the flow $\mathbb{S}\Phi$ on $\mathbb{S}%
\mathcal{V}$ and the flow $\mathbb{S}\Phi^{1}$ restricted to $\mathbb{S}%
_{P}^{0}\mathcal{V}$. Every chain recurrent component $_{\mathbb{S}}M_{j}$ is
mapped onto a chain recurrent component $e(_{\mathbb{S}}M_{j})$ and
$\overline{\pi_{P}(\mathcal{V}_{j}^{+})}\cap\mathbb{S}_{P}^{0}\mathcal{V=}%
e(_{\mathbb{S}}M_{j})$.
\end{lemma}

\begin{proof}
Obviously, $e$ is a conjugacy mapping chain recurrent components onto chain
recurrent components. The inclusion $\overline{\pi_{P}(\mathcal{V}_{j}^{+}%
)}\cap\mathbb{S}_{P}^{0}\mathcal{V\subset}e(_{\mathbb{S}}M_{j})$ holds since
for $v_{n}\in\mathcal{V}_{j}^{+}$ with $\left\Vert v_{n}\right\Vert
\rightarrow\infty$ for $n\rightarrow\infty$ one finds%
\[
\pi_{P}(v_{n})=\left(  \frac{v_{n}}{\left\Vert (v_{n},1)\right\Vert },\frac
{1}{\left\Vert (v_{n},1)\right\Vert }\right)  \text{ with }\frac{1}{\left\Vert
(v_{n},1)\right\Vert }\rightarrow0
\]
and every cluster point of $\frac{v_{n}}{\left\Vert (v_{n},1)\right\Vert }$ is
in $_{\mathbb{S}}M_{j}$. For the converse, let $(v,0)\in e(_{\mathbb{S}}%
M_{j})$. There are $v_{n}\in\mathcal{V}_{j}^{+}$ with $\left\Vert
v_{n}\right\Vert \rightarrow\infty$ and $\frac{v_{n}}{\left\Vert
(v_{n},1)\right\Vert }\rightarrow v$ for $n\rightarrow\infty$. Then the
assertion follows from $\pi_{P}(v_{n})\rightarrow(v,0)$ for $n\rightarrow
\infty$.
\end{proof}

Furthermore, the flow $\Phi$ on $\mathcal{V}$ induces a flow $\pi_{P}\Phi$ on
$\mathbb{S}_{P}^{+}\mathcal{V}$ given by%
\[
\pi_{P}\Phi(t,\pi_{P}v)=\left(  \frac{\Phi(t,v)}{\left\Vert (\Phi
(t,v),1)\right\Vert },\frac{1}{\left\Vert (\Phi(t,v),1)\right\Vert }\right)
\text{ for }t\in\mathbb{R},v\in\mathcal{V}.
\]

\begin{lemma}
\label{Lemma2}Let, for $\varepsilon>0$ small enough, $\zeta$ be an
$(\varepsilon,T)$-chain on $\mathbb{S}\mathcal{V}$ with lift $\widehat{\zeta}$
to $\mathcal{V}$. Then for $\alpha>0$ the associated chain $\pi_{P}%
(\alpha\widehat{\zeta})$ on $\mathbb{S}_{P}^{+}\mathcal{V}$ is an
$(\varepsilon,T)$-chain.
\end{lemma}

\begin{proof}
For $\zeta$ as in Definition \ref{Definition_lift}, the lifted chain
$\alpha\widehat{\zeta}$ is given by $\alpha w_{i}=\alpha C_{i-1}v_{i}$ for
$i=0,\ldots,n$. By Lemma \ref{Lemma1} it follows that $\left\Vert
w_{i}\right\Vert =C_{i-1}$ and $\Phi(T_{i},\alpha w_{i})=\alpha C_{i}%
\mathbb{S}\Phi(T_{i},v_{i})$. For $\varepsilon>0$ small enough, there is $a\in
A$ with%
\begin{align*}
d\left(  \mathbb{S}\Phi(T_{i},v_{i}),v_{i+1}\right)   &  =d_{a}\left(
\mathbb{S}\Phi(T_{i},v_{i}),v_{i+1}\right)  ,\\
d\left(  \pi_{P}\Phi(T_{i},\alpha w_{i}),\pi_{P}(\alpha w_{i+1})\right)   &
=d_{a}\left(  \pi_{P}\Phi(T_{i},\alpha w_{i}),\pi_{P}(\alpha w_{i+1})\right)
.
\end{align*}
One computes%
\begin{align*}
\pi_{P}\Phi(T_{i},\alpha w_{i})  &  =\left(  \frac{\Phi(T_{i},\alpha w_{i}%
)}{\left\Vert (\Phi(T_{i},\alpha w_{i}),1)\right\Vert },\frac{1}{\left\Vert
(\Phi(T_{i},\alpha w_{i}),1)\right\Vert }\right) \\
&  =\left(  \frac{\alpha C_{i}\mathbb{S}\Phi(T_{i},v_{i})}{\left\Vert (\alpha
C_{i}\mathbb{S}\Phi(T_{i},v_{i}),1)\right\Vert },\frac{1}{\left\Vert (\alpha
C_{i}\mathbb{S}\Phi(T_{i},v_{i}),1)\right\Vert }\right) \\
&  =\left(  \frac{\alpha C_{i}\mathbb{S}\Phi(T_{i},v_{i})}{\sqrt{\alpha
^{2}C_{i}^{2}+1}},\frac{1}{\sqrt{\alpha^{2}C_{i}^{2}+1}}\right) \\
&  =\frac{1}{\sqrt{\alpha^{2}C_{i}^{2}+1}}(\alpha C_{i}\mathbb{S}\Phi
(T_{i},v_{i}),1)
\end{align*}
and%
\[
\pi_{P}(\alpha w_{i+1})=\left(  \frac{\alpha w_{i+1}}{\left\Vert (\alpha
w_{i+1},1)\right\Vert },\frac{1}{\left\Vert (\alpha w_{i+1},1)\right\Vert
}\right)  =\left(  \frac{\alpha C_{i}v_{i+1}}{\sqrt{\alpha^{2}C_{i}^{2}+1}%
},\frac{1}{\sqrt{\alpha^{2}C_{i}^{2}+1}}\right)  .
\]
The local trivializations have, with $(b_{i}\cdot T_{i},y_{i})=\varphi
_{a}(\mathbb{S}\Phi(T_{i},v_{i}))$, the form%
\[
\varphi_{a}^{1}\left(  \pi_{P}\Phi(T_{i},\alpha w_{i})\right)  =\left(
b_{i}\cdot T_{i},\frac{\alpha C_{i}y_{i}}{\sqrt{\alpha^{2}C_{i}^{2}+1}}%
,\frac{1}{\sqrt{\alpha^{2}C_{i}^{2}+1}}\right)  ,
\]
and, with $(b_{i+1},x_{i+1})=\varphi_{a}(v_{i+1})$,%
\[
\varphi_{a}^{1}\left(  \pi_{P}(\alpha w_{i+1})\right)  =\left(  b_{i+1}%
,\frac{\alpha C_{i}x_{i+1}}{\sqrt{\alpha^{2}C_{i}^{2}+1}},\frac{1}%
{\sqrt{\alpha^{2}C_{i}^{2}+1}}\right)  .
\]
The distance $d_{a}$ in $\mathbb{S}\mathcal{V}^{1}\subset\mathcal{V}^{1}$ is
(cf. (\ref{d_a}))%
\[
d_{a}((v,r),(v^{\prime},r^{\prime}))=\max\{d(b,b^{\prime}),\left\Vert
(x,r)-(x^{\prime},r^{\prime})\right\Vert _{a}\},
\]
where $(v,r)=(\varphi_{a}^{1})^{-1}(b,(x,r))$ and $(v^{\prime},r^{\prime
})=(\varphi_{a}^{1})^{-1}(b^{\prime},(x^{\prime},r^{\prime}))$. It follows
that%
\begin{align*}
&  d_{a}\left(  \pi_{P}\Phi(T_{i},\alpha w_{i}),\pi_{P}(\alpha w_{i+1})\right)
\\
&  =\max\left\{  d(b_{i}\cdot T_{i},b_{i+1}),\frac{1}{\sqrt{\alpha^{2}%
C_{i}^{2}+1}}\left\Vert \left(  \alpha C_{i}y_{i},1\right)  -\left(  \alpha
C_{i}x_{i+1},1\right)  \right\Vert _{a}\right\} \\
&  \leq\max\left\{  d(b_{i}\cdot T_{i},b_{i+1}),\frac{\alpha C_{i}}%
{\sqrt{\alpha^{2}C_{i}^{2}+1}}\left\Vert y_{i}-x_{i+1}\right\Vert _{a}\right\}
\\
&  \leq\max\left\{  d(b_{i}\cdot T_{i},b_{i+1}),\left\Vert y_{i}%
-x_{i+1}\right\Vert _{a}\right\} \\
&  =d_{a}(\mathbb{S}\Phi(T_{i},v_{i}),v_{i+1}).
\end{align*}
This concludes the proof.
\end{proof}

The following result shows that one obtains chain transitive sets on the
closed upper hemisphere of the Poincar\'{e} sphere bundle from the chain
recurrent components on the sphere bundle $\mathbb{S}\mathcal{V}$.

\begin{theorem}
\label{Theorem_Poincare1}Let $\mathcal{V}_{j}^{+}=\{\alpha v\in\mathcal{V}%
\left\vert \alpha>0,v\in\,_{\mathbb{S}}M_{j}\right.  \}$ be the cone bundle on
$\mathcal{V}$ generated by a chain recurrent component $_{\mathbb{S}}%
M_{j},j\in\{1,\ldots,l_{1}\}$, on $\mathbb{S}\mathcal{V}$. Then the following
assertions are equivalent:

(a) The set $\overline{\pi_{P}\mathcal{V}_{j}^{+}}$ is chain transitive on the
closed upper hemisphere $\overline{\mathbb{S}_{P}^{+}\mathcal{V}}$ of the
Poincar\'{e} sphere bundle.

(b) The cone bundle $\mathcal{V}_{j}^{+}$ contains a half-line $l=\{\alpha
v_{0}\left\vert \alpha>0\right.  \}$ for some $v_{0}\in\,_{\mathbb{S}}M_{j}$
such that $\pi_{P}l$ is chain transitive on $\mathbb{S}_{P}^{+}\mathcal{V}$.
\end{theorem}

\begin{proof}
It is clear that (a) implies (b). For the converse let $z_{1},z_{2}%
\in\mathcal{V}_{j}^{+}$, hence there are $\alpha_{1},\alpha_{2}>0$ and
$v_{1},v_{2}\in\,_{\mathbb{S}}M_{j}$ with $z_{1}=\alpha_{1}v_{1}$ and
$z_{2}=\alpha_{2}v_{2}$. Fix $\varepsilon,T>0$. We construct an $(\varepsilon
,T)$-chain from $\pi_{P}z_{1}$ to $\pi_{P}z_{2}$. There is an $(\varepsilon
,T)$-chain $\zeta^{1}$ in $_{\mathbb{S}}M_{j}$ from $v_{1}$ to $v_{0}$, hence
the induced chain $\alpha_{1}\widehat{\zeta^{1}}$ in $\mathcal{V}$ starts in
$\alpha_{1}v_{0}=z_{1}$ and ends in $\alpha_{1}\gamma_{1}v_{0}\in l$ for some
$\gamma_{1}>0$. Since $v_{0},v_{2}\in\mathrm{\,}_{\mathbb{S}}M_{j}$ one finds
an $(\varepsilon,T)$-chain $\zeta^{2}$ in $_{\mathbb{S}}M_{j}$ from $v_{0}$ to
$v_{2}$. The induced chain $\widehat{\zeta^{2}}$ in $\mathcal{V}$ ends in
$\gamma_{2}v_{2}$ for some $\gamma_{2}>0$. This yields an $(\varepsilon
,T)$-chain $\frac{\alpha_{2}}{\gamma_{2}}\widehat{\zeta^{2}}$ in $\mathcal{V}$
from $\frac{\alpha_{2}}{\gamma_{2}}v_{0}$ to $\frac{\alpha_{2}}{\gamma_{2}%
}\gamma_{2}v_{2}=\alpha_{2}v_{2}=z_{2}$.

By chain transitivity of $\pi_{P}l$ one finds an $(\varepsilon,T)$-chain
$\zeta^{0}$ in $\mathbb{S}_{P}^{+}\mathcal{V}$ from $\pi_{P}(\alpha_{1}%
\gamma_{1}v_{0})\in\pi_{P}l$ to $\pi_{P}(\frac{\alpha_{2}}{\gamma_{2}}%
v_{0})\in\pi_{P}l$. By Lemma \ref{Lemma2} the $(\varepsilon,T)$-chains
$\zeta^{1}$ and $\zeta^{2}$ on $\mathbb{S}\mathcal{V}$ induce $(\varepsilon
,T)$-chains $\pi_{P}\left(  \frac{\alpha_{2}}{\gamma_{2}}\widehat{\zeta^{2}%
}\right)  $ and $\pi_{P}\left(  \alpha_{1}\widehat{\zeta^{1}}\right)  $ on
$\mathbb{S}_{P}^{+}\mathcal{V}$. Then one obtains an $(\varepsilon,T)$-chain
from $\pi_{P}z_{1}$ to $\pi_{P}z_{2}$ by the concatenation%
\[
\pi_{P}\left(  \frac{\alpha_{2}}{\gamma_{2}}\widehat{\zeta^{2}}\right)
\circ\,\zeta^{0}\circ\pi_{P}\left(  \alpha_{1}\,\widehat{\zeta^{1}}\right)  .
\]
Thus $\pi_{P}\mathcal{V}_{j}^{+}$ is chain transitive which implies that also
its closure $\overline{\pi_{P}\mathcal{V}_{j}^{+}}$ is chain transitive.
\end{proof}

Next we discuss when the existence of a half-line $l$ as above can be
guaranteed. It is clear that a necessary condition is that $0$ is in the Morse
spectrum. Conversely, Theorem \ref{Theorem_Poincare2} will show that $l$
exists if $0$ is in the interior of the Morse spectrum.

The following lemma will be needed in Step 2 of the next proof.

\begin{lemma}
\label{Lemma_dio}Let $a,b,c$ be real numbers with $a,b,c>1$ and $\frac{\log
b}{\log a}\in\mathbb{R}\setminus\mathbb{Q}$. Then for every $\delta>0$ there
are $k,\ell\in\mathbb{N}$ such that $\left\vert a^{k}b^{-\ell}-c\right\vert
<\delta$.
\end{lemma}

\begin{proof}
Since the logarithm is continuously invertible, it suffices to show that for
every $\delta>0$ there are $k,\ell\in\mathbb{N}$ with
\[
\delta>\left\vert \log(a^{k}b^{-\ell})-\log c\right\vert =\left\vert k\log
a-\ell\log b-\log c\right\vert ,
\]
or, dividing by $\log a>0$,%
\begin{equation}
\left\vert k-\ell\frac{\log b}{\log a}-\frac{\log c}{\log a}\right\vert
<\frac{\delta}{\log a}. \label{K1}%
\end{equation}
Since $\frac{\log b}{\log a}$ is irrational, Kronecker's theorem (cf. Cassels
\cite[Theorem IV, p.53]{Cass57}) implies that for any $\delta>0$ there are
$k,\ell\in\mathbb{N}$ satisfying (\ref{K1}).
\end{proof}

\begin{remark}
The proof above uses Kronecker's theorem (1884) \cite{Kron84}. The special
case of this theorem needed here also follows from an earlier result by
Tchebychef (1866) \cite[Th\'{e}or\`{e}me, p. 679]{Tcheb}; cf. Colonius,
Santana, and Setti \cite[Lemma 3.4]{ColRS}, where Lemma \ref{Lemma_dio} is
used to prove an approximate controllability result for bilinear control
systems, using a different spectral notion.
\end{remark}

\begin{lemma}
\label{Proposition_cone}Consider a chain recurrent component $_{\mathbb{S}%
}M_{j}$ on the sphere bundle $\mathbb{S}\mathcal{V}$ and assume that there is
$\delta_{0}>0$ such that for all $\varepsilon,T>0$ there are periodic
$(\varepsilon,T)$\textit{-chains} $\zeta^{\pm}$ in $_{\mathbb{S}}M_{j}$ with
chain exponents $\lambda(\zeta^{+})>\delta_{0}$ and $\lambda(\zeta
^{-})<-\delta_{0}$. Then the cone bundle $\mathcal{V}_{j}^{+}:=\{\alpha
v\in\mathcal{V}\left\vert \alpha>0,v\in\,_{\mathbb{S}}M_{j}\right.  \}$
contains a half-line $l=\{\alpha v\left\vert \alpha>0\right.  \}$ for some
$v\in\,_{\mathbb{S}}M_{j}$ such that $\pi_{P}l$ is chain transitive.
\end{lemma}

\begin{proof}
\textbf{Step 1}. First we show that one may choose the initial points $v^{\pm
}$ of $\zeta^{\pm}$ independently of $\varepsilon,T$.

Fix any point $v\in\,_{\mathbb{S}}M_{j}$. Let $\varepsilon,T>0$ and consider
for $S\geq T$ periodic $(\varepsilon,S)$-chains $\zeta^{\pm}$ with
$\lambda(\zeta^{+})>\delta_{0}$ and $\lambda(\zeta^{-})<-\delta_{0}$. The
proof is given for $\zeta^{+}$, the proof for $\zeta^{-}$ is analogous and
will be omitted. It suffices to prove that for every $\delta>0$ there exists a
periodic $(\varepsilon,T)$-chain $\zeta$ through $v$ with $\left\vert
\lambda(\zeta^{+})-\lambda(\zeta)\right\vert \,<\delta$.

Let $\zeta^{+}$ be given by $T_{0},\ldots,T_{n}\geq S,v_{0}=v^{+},v_{1}%
,\ldots,v_{n}=v^{+}\in\,_{\mathbb{S}}M_{j}$ with total time $\sigma$ and%
\[
d(\Phi(T_{i},v_{i}),v_{i+1})<\varepsilon\text{ for }i=0,\ldots,n-1.
\]
By Colonius and Kliemann \cite[Lemma B.2.23]{ColK00} there exists $\bar
{T}(\varepsilon,T)>0$ such that for all $v^{\prime},v^{\prime\prime}%
\in\,_{\mathbb{S}}M_{j}$ there is an $(\varepsilon,T)$-chain from $v^{\prime}$
to $v^{\prime\prime}$ with total time equal to or less than $\bar
{T}(\varepsilon,T)$.

Choose $(\varepsilon,T)$-chains $\zeta^{1}$ from $v$ to $v^{+}$ and $\zeta
^{2}$ from $v^{+}$ to $v$ on $\mathbb{S}\mathcal{V}$. For $k=1,2$ let them be
given by $v_{0}^{k},\ldots,v_{n_{k}}^{k}$ and times $T_{0}^{k},\,\ldots
,T_{n_{k}-1}^{k}>T$ with total time $\sigma_{k}\leq\bar{T}(\varepsilon,T)$ and
chain exponent%
\[
\lambda(\zeta^{k})=\frac{1}{\sigma_{k}}\sum_{i=0}^{n_{k}-1}\log\left\Vert
\Phi(T_{i}^{k},v_{i}^{k})\right\Vert \,.
\]
By compactness of the sphere bundle and continuity together with $\sigma_{k}$
$\leq\bar{T}(\varepsilon,T)$ there is a constant $c_{0}=c_{0}(\varepsilon
,T)>0$ such that%
\[
\sum_{i=0}^{n_{k}-1}\log\left\Vert \Phi(T_{i}^{k},v_{i}^{k})\right\Vert \leq
c_{0}.
\]
The concatenation $\zeta^{2}\circ\zeta^{+}\circ\zeta^{1}$ is a periodic
$(\varepsilon,T)$-chain through $v$ with chain exponent
\begin{align*}
&  \lambda(\zeta^{2}\circ\zeta^{+}\circ\zeta^{1})\\
&  =\frac{1}{\sigma+\sigma_{1}+\sigma_{2}}\left[  \sum_{i=0}^{n-1}%
\log\left\Vert \Phi(T_{i},v_{i})\right\Vert +\sum_{i=0}^{n_{1}-1}%
\log\left\Vert \Phi(T_{i}^{1},v_{i}^{1})\right\Vert +\sum_{i=0}^{n_{2}-1}%
\log\left\Vert \Phi(T_{i}^{2},v_{i}^{2})\right\Vert \right]  .
\end{align*}
Since $\frac{1}{\sigma+\sigma_{1}+\sigma_{2}}-\frac{1}{\sigma}=\frac
{\sigma_{1}+\sigma_{2}}{\sigma+\sigma_{1}+\sigma_{2}}\cdot\frac{1}{\sigma}$ it
follows that%
\[
\left(  \frac{1}{\sigma+\sigma_{1}+\sigma_{2}}-\frac{1}{\sigma}\right)
\sum_{i=0}^{n-1}\log\left\Vert \Phi(T_{i},v_{i})\right\Vert =\frac{\sigma
_{1}+\sigma_{2}}{\sigma+\sigma_{1}+\sigma_{2}}\lambda(\zeta^{+}).
\]
Thus one may choose $S$ and hence $\sigma$ large enough such that as desired%
\[
\left\vert \lambda(\zeta^{+})-\lambda(\zeta^{2}\circ\zeta^{+}\circ\zeta
^{1})\right\vert \,<\delta.
\]

\textbf{Step 2. }For $v\in$\textbf{ }$_{\mathbb{S}}M_{j}$ the projection
$\pi_{P}l$ of the half-line $l:=\{\alpha v\in\mathcal{V}\left\vert
\alpha>0\right.  \}$ is chain transitive.

We construct for all $\alpha_{0}v,\alpha_{1}v\in l$ and every $\varepsilon
,T>0$ an $(\varepsilon,T)$-chain from $\pi_{P}(\alpha_{0}v)$ to $\pi
_{P}(\alpha_{1}v)$. The strategy for the proof is to go $k$ times through the
periodic chain $\zeta^{+}$ and $\ell$ times through the periodic chain
$\zeta^{-}$, and to adjust the numbers $k,\ell\in\mathbb{N}$ such that in
$\mathcal{V}$ one approaches $\alpha_{1}v$. Fix $\varepsilon,T>0$, add a
superscript $\pm$ to the components of the chain $\zeta^{\pm}$, and abbreviate%
\begin{equation}
\beta^{\pm}=\prod_{i=0}^{n^{\pm}-1}\left\Vert \Phi(T_{i}^{\pm},v_{i}^{\pm
})\right\Vert . \label{beta}%
\end{equation}
Define for $k,\ell\in\mathbb{N}$ a periodic $(\varepsilon,T)$-chain
$\zeta^{k,\ell}$ through $v$ on $\mathbb{S}\mathcal{V}$ by the concatenation%
\[
\zeta^{k,\ell}=(\zeta^{-})^{\ell}\circ(\zeta^{+})^{k}.
\]
The endpoint of the lift $\widehat{\zeta^{k,\ell}}$ of $\zeta^{k,\ell}$ to
$\mathcal{V}$ is given by $\left(  \beta^{-}\right)  ^{\ell}\left(  \beta
^{+}\right)  ^{k}v$. The chain $\alpha_{0}\widehat{\zeta^{k,\ell}}$ starts in
$\alpha_{0}v$ and ends in $\alpha_{0}w_{n}=\left(  \beta^{+}\right)
^{k}\left(  \beta^{-}\right)  ^{\ell}\alpha_{0}v$.

Lemma \ref{Lemma_dio} implies the following: If one can choose $\beta^{+}$ and
$\beta^{-}$ such that $\frac{\log\beta^{-}}{\log\beta^{+}}$ is irrational then
for every $\delta>0$ there are $k,\ell\in\mathbb{N}$ with%
\[
\left\vert \left(  \beta^{+}\right)  ^{k}\left(  \beta^{-}\right)  ^{\ell
}-\frac{\alpha_{1}}{\alpha_{0}}\right\vert <\delta.
\]
Hence one can choose $k,\ell$ such that%
\[
\left\Vert \left(  \beta^{-}\right)  ^{\ell}\left(  \beta^{+}\right)
^{k}\alpha_{0}v-\alpha_{1}v\right\Vert <\alpha_{0}\delta<\varepsilon.
\]
This yields a $(2\varepsilon,T)$-chain $\zeta^{k,\ell}$ on $\mathbb{S}%
\mathcal{V}$ from $\alpha_{0}v$ to $\alpha_{1}v$ such that, by Lemma
\ref{Lemma2}, $\pi_{P}\left(  \alpha_{0}\widehat{\zeta^{k,\ell}}\right)  $ is
a $(2\varepsilon,T)$-chain from $\pi_{P}(\alpha_{0}v)$ to $\pi_{P}(\alpha
_{1}v)$.

Since $\varepsilon,T>0$ are arbitrary, this completes Step 2 once we have
justified the assumption that we can choose $\beta^{+}$ and $\beta^{-}$ as in
(\ref{beta}) such that $\frac{\log\beta^{-}}{\log\beta^{+}}$ is irrational. If
this is not the case, we will adjust $v_{1}^{+}$ appropriately. First note
that we may assume for the length of the periodic chain $\zeta^{+}$ that
$n^{+}>1$ (otherwise we replace it by $\zeta^{+}\circ\zeta^{+}$). The map
$\Phi(T_{1}^{+},\cdot)$ restricted to the fiber $\mathcal{V}_{\pi(v)}$ \ with
range $\mathcal{V}_{\pi(v)\cdot T_{1}^{+}}$ is a linear isomorphism. It maps
any neighborhood of $v_{1}^{+}$ onto a neighborhood of $\Phi(T_{1}^{+}%
,v_{1}^{+})$. Thus one finds arbitrarily close to $v_{1}^{+}$ points $v_{1}%
\in\mathcal{V}_{\pi(v)}$ such that for%
\[
\beta_{0}:=\left\Vert \Phi(T_{1},v_{1})\right\Vert \prod_{i=0,i\not =1}%
^{n^{+}-1}\left\Vert \Phi(T_{i}^{+},v_{i}^{+})\right\Vert
\]
the quotient $\frac{\log\beta^{-}}{\log\beta_{0}}$is irrational. Choosing the
neighborhood $N$ small enough we may replace $v_{1}^{+}$ by $v_{1}$ without
destroying the property of a $(2\varepsilon,T)$-chain from $\alpha_{0}v$ to
$\alpha_{1}v$.
\end{proof}

The following theorem provides a sufficient condition for the property that
the cone bundle induced by a central subbundle in the Selgrade decomposition
yields a chain transitive set on the upper hemisphere of the Poincar\'{e}
sphere bundle.

\begin{theorem}
\label{Theorem_Poincare2}If a cone bundle $\mathcal{V}_{j}^{+}$ generated by a
chain recurrent component $_{\mathbb{S}}M_{j},j\in\{1,\ldots,l_{1}\}$, on
$\mathbb{S}\mathcal{V}$ satisfies $0\in\mathrm{int}\left(  \Sigma
_{Mo}(\mathcal{V}_{j}^{+})\right)  $, then the set $\overline{\pi
_{P}\mathcal{V}_{j}^{+}}$ is chain transitive on the closed upper hemisphere
$\overline{\mathbb{S}_{P}^{+}\mathcal{V}}$ of the Poincar\'{e} sphere bundle.
\end{theorem}

\begin{proof}
The Morse spectrum over a chain recurrent component in the projective bundle
is a bounded interval, cf. \cite[Proposition 6.2.14]{ColK00}. By Proposition
\ref{Proposition_Morse} the same is valid for the Morse spectrum $\Sigma
_{Mo}(_{\mathbb{S}}M_{i})$. If $0\in\mathrm{int}\left(  \Sigma_{Mo}%
(_{\mathbb{S}}M_{i})\right)  $, it follows that there is $\delta_{0}>0$ such
that for all $\varepsilon,T>0$ there are periodic $(\varepsilon,T)$-chains
$\zeta^{+},\zeta^{-}$ on $\mathrm{\,}_{\mathbb{S}}M_{i}$, with $\lambda
(\zeta^{+})>\delta_{0}$ and $\lambda(\zeta^{-})<-\delta_{0}$ for some
$\delta_{0}>0$. Thus Lemma \ref{Proposition_cone} implies that $\mathcal{V}%
_{j}^{+}$ contains a half-line $l=\{\alpha v\left\vert \alpha>0\right.  \}$
for some $v\in\,_{\mathbb{S}}M_{j}$ such that $\pi_{P}l$ is chain transitive
on $\mathbb{S}_{P}^{+}\mathcal{V}$. Now the assertion follows from Theorem
\ref{Theorem_Poincare1}.
\end{proof}

\section{Examples\label{Section5}}

In this section the results above are illustrated by autonomous linear
differential equations and by bilinear control systems.

\begin{example}
Consider
\[
\left(
\begin{array}
[c]{c}%
\dot{x}\\
\dot{y}%
\end{array}
\right)  =\left(
\begin{array}
[c]{cc}%
1 & 0\\
0 & 0
\end{array}
\right)  \left(
\begin{array}
[c]{c}%
x\\
y
\end{array}
\right)  .
\]
Here $_{\mathbb{S}}M^{\pm}=\{(0,\pm1)\}\subset\mathbb{S}^{1}$ are chain
recurrent components generating the cones $V^{\pm}=\{(0,\pm\alpha)\left\vert
\alpha>0\right.  \}$. The solutions in $\mathbb{R}^{2}$ in the eigenspace
$\{0\}\times\mathbb{R}$ for the eigenvalue $0$ are equilibria, hence $V^{\pm}$
form chain transitive sets in $\mathbb{R}^{2}$. The Morse spectrum is given by
$\Sigma_{Mo}(V^{\pm})=\{0\}$ and on the closed upper hemisphere $\overline
{\mathbb{S}^{2,+}}$ one obtains the chain transitive sets%
\begin{align*}
\overline{\pi_{P}V^{+}}  &  =\{(0,s_{2},s_{3})\in\mathbb{S}^{2}\left\vert
s_{2}\geq0,s_{3}\geq0\right.  \},\\
\pi_{P}V^{-}  &  =\{(0,s_{2},s_{3})\in\mathbb{S}^{2}\left\vert s_{2}%
\leq0,s_{3}\geq0\right.  \}.
\end{align*}
The intersection with the equator $\mathbb{S}^{2,0}=\{(s_{1},s_{2},s_{3}%
)\in\mathbb{S}^{2}\left\vert s_{3}=0\right.  \}$ is given by $\overline
{\pi_{P}V^{\pm}}\cap$ $\mathbb{S}^{2,0}=(0,\pm1,0)=e((0,\pm1))$.

For the system%
\[
\left(
\begin{array}
[c]{c}%
\dot{x}\\
\dot{y}%
\end{array}
\right)  =\left(
\begin{array}
[c]{cc}%
0 & 1\\
0 & 0
\end{array}
\right)  \left(
\begin{array}
[c]{c}%
x\\
y
\end{array}
\right)  .
\]
the unit sphere $S^{1}$ is a chain recurrent component. The solutions in
$\mathbb{R}^{2}$ are unbounded, hence the system on $\mathbb{R}^{2}$ has no
recurrence properties. The compactification provided by the Poincar\'{e}
sphere yields the chain transitive set%
\[
\overline{\pi_{P}(\mathbb{R}^{2})}=\left\{  (s_{1},s_{2},s_{3})\in
\mathbb{S}^{2}\left\vert s_{3}\geq0\right.  \right\}  .
\]

\end{example}

For general autonomous linear differential equations the assumption in Theorem
\ref{Theorem_Poincare2} that $0$\textrm{\ }is in the interior of the Morse
spectrum is never satisfied, since the Morse spectrum reduces to the real
parts of the eigenvalues. A direct application of Theorem
\ref{Theorem_Poincare1} yields the following result.

\begin{corollary}
\label{Corollary1}For an autonomous linear differential equation given by
$\dot{x}=Ax$ with $A\in\mathbb{R}^{d\times d}$, suppose that the subspace%
\[
V:=%
%TCIMACRO{\dbigoplus }%
%BeginExpansion
{\displaystyle\bigoplus}
%EndExpansion
\{E(\mu)\left\vert \mu\in\mathrm{spec}(A):\operatorname{Re}\mu=0\right.  \},
\]
is nontrivial; here $E(\mu)$ is the (real) generalized eigenspace for the
eigenvalue $\mu$.

Then the projection of $V$ to projective space $\mathbb{P}^{d-1}$ is a chain
recurrent component $_{\mathbb{P}}M$ and there are one or two cones $V^{+}$
and $V^{-}$ such that $V^{\pm}\cap\mathbb{S}^{d-1}$are chain recurrent
components on the unit sphere $\mathbb{S}^{d-1}$ which project onto
$_{\mathbb{P}}M$. The closure of the projection $\pi_{P}V^{\pm}$ is chain
transitive for the induced flow on the closed upper hemisphere $\overline
{\mathbb{S}_{P}^{+}\mathbb{R}^{d}}=\overline{\mathbb{S}^{d,+}}$ of the
Poincar\'{e} sphere.
\end{corollary}

\begin{proof}
It is well known that the projection of $V$ is a chain recurrent component
$_{\mathbb{P}}M$ in projective space $\mathbb{P}^{d-1}$ (cf., e.g., Colonius
and Kliemann \cite[Theorem 4.1.3]{ColK14} for a proof). Then it easily follows
that there are one or two chain recurrent components on the unit sphere which
project onto $_{\mathbb{P}}M$. The last assertion of the corollary follows
from Theorem \ref{Theorem_Poincare1}, since for an eigenvalue $\mu$ with
$\operatorname{Re}\mu=0$ the (real) eigenspace $E(\mu)$ consists of a
continuum of periodic solutions and hence contains a chain transitive
half-line. Thus also the projection to $\mathbb{S}_{P}^{+}\mathbb{R}^{d}$ is
chain transitive.
\end{proof}

Another class of linear flows is given by homogeneous bilinear control systems
of the form
\begin{equation}
\dot{x}(t)=A_{0}x(t)+\sum_{i=1}^{m}u_{i}(t)A_{i}x(t),\quad u(t)\in\Omega,
\label{control}%
\end{equation}
where $A_{0},A_{1},\ldots,A_{m}\in\mathbb{R}^{d\times d}$. The control
functions $u=(u_{1},\ldots,u_{m})$ have values in a convex and compact subset
$\Omega\subset\mathbb{R}^{m}$ and the set of admissible controls is
$\mathcal{U}=\{u\in L^{\infty}(\mathbb{R},\mathbb{R}^{m})\left\vert
u(t)\in\Omega\text{ for almost all }t\right.  \}$. Denote the solutions of
(\ref{control}) with initial condition $x(0)=x$ by $\varphi(t,x,u),t\in
\mathbb{R}$.

A control system of this form defines a linear flow, the control flow, on the
vector bundle $\mathcal{V}=\mathcal{U}\times\mathbb{R}^{d}$ given by
$\Phi:\mathbb{R}\times\mathcal{U}\times\mathbb{R}^{d}\rightarrow
\mathcal{U}\times\mathbb{R}^{d},\,(t,u,x)\mapsto(u(t+\cdot),\varphi(t,x,u))$
where $u(t+\cdot)(s):=u(t+s),s\in\mathbb{R}$, is the right shift and
$\mathcal{U}\subset L^{\infty}(\mathbb{R},\mathbb{R}^{m})$ is considered in a
metric for the weak$^{\ast}$ topology. Then $\mathcal{U}$ is compact and chain
transitive; cf. Colonius and Kliemann \cite[Chapter 4]{ColK00} or Kawan
\cite[Section 1.4]{Kawa}.

The following two-dimensional example is a minor modification of \cite[Example
5.5.12]{ColK00}. Here four chain recurrent components on the sphere bundle are
present, and the corresponding cone bundles contain $0$ in the interior of the
Morse spectrum.

\begin{example}
\label{MORSE:ex47}Consider the control system
\[
\left[
\begin{array}
[c]{c}%
\dot{x}_{1}\\
\dot{x}_{2}%
\end{array}
\right]  =\left[  \left(
\begin{array}
[c]{cc}%
0 & -\frac{1}{4}\\
-\frac{1}{4} & 0
\end{array}
\right)  +u_{1}\left(
\begin{array}
[c]{cc}%
1 & 0\\
0 & 1
\end{array}
\right)  +u_{2}\left(
\begin{array}
[c]{cc}%
0 & 1\\
0 & 0
\end{array}
\right)  +u_{3}\left(
\begin{array}
[c]{cc}%
0 & 0\\
1 & 0
\end{array}
\right)  \right]  \left[
\begin{array}
[c]{c}%
x_{1}\\
x_{2}%
\end{array}
\right]
\]
with
\[
u(t)=(u_{1}(t),u_{2}(t),u_{3}(t))\in\Omega=\left[  -1,1\right]  \times\left[
-1/4,1/4\right]  \times\left[  -1/4,1/4\right]  \subset\mathbb{R}^{3}\text{.}%
\]
This defines a linear flow $\Phi$ on the vector bundle $\mathcal{U}%
\times\mathbb{R}^{2}$ given by $\Phi(t,u,x)=\left(  u(t+\cdot),\varphi
(t,x,u)\right)  $. With%
\begin{align*}
A_{1}  &  :=\left\{  \left(  x_{1},x_{2}\right)  \in\mathbb{R}^{2}\left\vert
x_{2}=\alpha x_{1},\alpha\in\lbrack-\sqrt{2},-1/\sqrt{2}]\right.  \right\}
,\\
A_{2}  &  :=\left\{  \left(  x_{1},x_{2}\right)  \in\mathbb{R}^{2}\left\vert
x_{2}=\alpha x_{1},\alpha\in\lbrack1/\sqrt{2},\sqrt{2}]\right.  \right\}
\end{align*}
the two chain recurrent components of $\mathbb{P}\Phi$ on the projective
bundle $\mathcal{U}\times\mathbb{P}^{1}$ are%
\[
_{\mathbb{P}}M_{i}=\left\{  (u,p)\in\mathcal{U}\times\mathbb{P}^{1}\left\vert
\mathbb{P}\varphi(t,p,u)\in\pi_{\mathbb{P}}A_{i}\text{ for }t\in
\mathbb{R}\right.  \right\}  ,\,i=1,2.
\]
With%
\[
A_{i}^{+}=A_{i}\cap\left(  (0,\infty)\times\mathbb{R}\right)  ,A_{i}^{-}%
=A_{i}\cap\left(  (-\infty,0)\times\mathbb{R}\right)  ,
\]
the four chain recurrent components for $\mathbb{S}\Phi$ on the sphere bundle
$\mathbb{S}(\mathcal{U}\times\mathbb{R}^{2})=\mathcal{U}\times\mathbb{S}^{1}$
are given by%
\[
_{\mathbb{S}}M_{i}^{\pm}=\left\{  (u,s)\in\mathcal{U}\times\mathbb{S}%
^{1}\left\vert \mathbb{S}\varphi(t,s,u)\in\pi_{\mathbb{S}}A_{i}^{\pm}\text{
for }t\in\mathbb{R}\right.  \right\}  ,\,i=1,2.
\]
One computes the Morse spectra of the generated cone bundles $\mathcal{V}%
_{i}^{\pm}$ as (here they are determined by the eigenvalues for constant
controls)%
\[
\mathbf{\Sigma}_{Mo}(\mathcal{V}_{1}^{\pm})=\left[  -2,1/2\right]  \text{ and
}\mathbf{\Sigma}_{Mo}(\mathcal{V}_{2}^{\pm})=\left[  -1/2,2\right]  .
\]
Hence for each of them the Morse spectrum contains $0$ in the interior. By
Theorem \ref{Theorem_Poincare2} the projections $\pi_{P}\mathcal{V}_{i}^{\pm}$
are chain transitive on the upper hemisphere $\mathbb{S}_{P}^{+}%
(\mathcal{U}\times\mathbb{R}^{2})=\mathcal{U}\times\mathbb{S}^{2,+}$ of the
Poincar\'{e} sphere bundle.
\end{example}

\textbf{Acknowledgement.} I am indebted to Alexandre Santana for helpful
discussions concerning the Poincar\'{e} sphere and to Mauro Patr\~{a}o for a
hint to Kronecker's theorem.

\end{document}